\documentclass[12pt]{amsart}

\usepackage{amssymb,latexsym, color}
\usepackage{enumerate}

\makeatletter
\@namedef{subjclassname@2010}{%
  \textup{2010} Mathematics Subject Classification}
\makeatother

\theoremstyle{plain}
\newtheorem{theorem}{Theorem}
\newtheorem{corollary}{Corollary}

\newtheorem{lemma}{Lemma}
\theoremstyle{definition}
\newtheorem{definition}{Definition}

\def\PP{\mathbb{P}^}
\def\sr{\mathrm{sr}}
\def\sbr{\mathrm{sbr}}

\date{}

\begin{document}

\baselineskip=17pt

\title[Unique Symmetric Decomposition]
{Unique decomposition for a polynomial of low rank}

\author[E. Ballico]{Edoardo Ballico}
\address{Dept. of Mathematics\\
  University of Trento\\
38123 Povo (TN), Italy}
\email{edoardo.ballico@unitn.it}

\author[A. Bernardi]{Alessandra Bernardi}
\address{Dipartimento di Matematica ``Giuseppe Peano'', Universit\`a degli Studi di Torino,
Via Carlo Alberto 10, 
I-10123 Torino, Italy
.}
\email{alessandra.bernardi@unito.it}

\date{}

\thanks{The authors were partially supported by CIRM of FBK Trento 
(Italy), Project Galaad of INRIA Sophia Antipolis M\'editerran\'ee 
(France),  Marie Curie: Promoting science (FP7-PEOPLE-2009-IEF), MIUR and GNSAGA of 
INdAM (Italy).}

\begin{abstract}Let $F$ be a homogeneous polynomial of 
degree $d$ in $m+1$ variables defined over an algebraically closed 
field of characteristic 0 and suppose that $F$ belongs to the 
$s$-th secant variety of the $d$-uple Veronese embedding
of $\mathbb{P}^m$ into  $ \PP {{m+d\choose d}-1}$ but that its 
minimal decomposition as a sum of $d$-th powers of linear forms requires more than $s$ addenda. We show that if $s\leq d$ then  $F$ can be uniquely written as $F=M_1^d+\cdots + M_t^d+Q$, where
$M_1, \ldots , M_t$ are linear forms with $t\leq (d-1)/2$, and $Q$ a binary form  such that $Q=\sum_{i=1}^q l_i^{d-d_i}m_i$ with $l_i$'s linear forms and $m_i$'s forms of degree $d_i$ such that $\sum (d_i+1)=s-t$.
\end{abstract}


\maketitle

\section*{Introduction}

In this paper we will always work with an algebraically closed field $K$ of characteristic 0. Let $X_{m,d}\subset \mathbb{P}^N$, with $m \ge 1$,  $d \ge 2$ and 
$N:= {{m+d}\choose m}-1$, be the classical Veronese variety obtained 
as the image of the $d$-uple Veronese embedding $\nu _d: \PP m \to \PP 
N$. The {\it $s$-th secant variety} $\sigma_s(X_{m,d})$ of 
Veronese variety $X_{m,d}$ is the Zariski closure in $\mathbb {P}^N$
of the union of all linear spans $\langle P_1, \ldots , P_s \rangle$ 
with $P_1, \ldots , P_s \in X_{m,d}$.
For  any point $P\in \PP N$, we indicate with $\sbr(P)$ the 
minimum integer $s$ such that $P\in \sigma_s(X_{m,d})$. This integer 
is called the {\it symmetric border rank} of $P$.

Since  $\mathbb{P}^m\simeq \mathbb{P}(K[x_0, \ldots , 
x_m]_1)\simeq \mathbb{P}(V^*)$, with $V$ an $(m+1)$-dimensional vector space over $K$, the generic element belonging to 
$\sigma_s(X_{m,d})$ is  the projective class of a form (a symmetric tensor) 
of type:
\begin{equation}\label{FT}
F=L_1^d+ \cdots + L_r^d, \; \; \; \;   (T=v_1^{\otimes d}+ \cdots 
+v_r^{\otimes d}).
\end{equation}
The minimum $r\in \mathbb{N}$ such that $F$  can be written as in (\ref{FT}) is the {\it symmetric rank} of $F$ and we denote it $\sr(F)$ ($\sr(T)$, if we replace $F$ with $T$).

The decomposition of a homogeneous polynomial that combines a minimum 
number of terms and that involves a minimum number of variables is a 
problem that is having a great deal of attentions not only from classical Algebraic Geometry (\cite{ah}, \cite{ik}, \cite{cc}, \cite{cmr}, \cite{lt}), but also from applications like Computational Complexity (\cite{ls}) 
and Signal Processing (\cite{vw}).
 
At the Workshop on Tensor Decompositions and Applications (September 13--17, 2010,
Monopoli, Bari, Italy), A. Bernardi presented a work in collaboration with E. Ballico where a possible structure of small rank homogeneous polynomials with border rank smaller than the rank was characterized (see \cite{bb1}).
It is well known that, if a homogeneous polynomial $F$ is such that $\sbr(F)<\sr(F)$, then there are infinitely many  decompositions of $F$ as in (\ref{FT}). Our purpose in \cite{bb1} was to find, among all the possible decompositions of $F$, a ``best" one in terms of number of variables. Namely: Does there exist a canonical choice of two variables  such that most of the terms involved in the decomposition (\ref{FT}) of $F$ depend only on those two variables? The precise statement of that result  is the following:
\\
\\
(\cite{bb1}, Corollary 1)
 Let $F\in K[x_0, \ldots , x_m]_d$ be such that
$\sbr(F)+\sr(F)\leq 2d+1$ and $\sbr(F)<\sr(F)$. Then there are an integer $t \ge 0$, linear forms $L_1, L_2,M_1, \ldots , 
M_t\in K[x_0, \ldots , x_m]_1$,  and a form $Q\in K[L_1,L_2]_d$ such that 
$F=Q+M_1^d+\cdots + M_t^d$, $t \le \sbr(F)+\sr(F)-d-2$, and $\sr(F)=\sr(Q)+t$.
Moreover $t$, $M_1,\dots , M_t$ and the linear span of $L_1, L_2$ are 
uniquely determined by $F$. 
\\
\\
In terms of tensors it can be translated as follows:
\\
\\
(\cite{bb1}, Corollary 2) Let $T\in S^dV^*$ be such that
$\sbr(T)+\sr(T)\leq 2d+1$ and $\sbr(T)<\sr(T)$. Then there are an integer $t\ge 0$, vectors $v_1, v_2,w_1, \ldots , 
w_t\in S^1V^*$, and a symmetric tensor $v\in S^d(\langle v_1, v_2 \rangle)$ such that
$T=v+w_1^{\otimes d}+\cdots + w_t^{\otimes d}$, $t\le \sbr(T)+\sr(T)-d-2$, and $\sr(T)=\sr(v)+t$.
Moreover $t$, $w_1,\dots ,w_t$ and   $\langle v_1, v_2\rangle $ are 
uniquely determined by $T$.
\\
\\
The natural questions that arose from applied people at the workshop in Monopoli mentioned above, were about the possible uniqueness of the binary form $Q$ in \cite{bb1}, Corollary 1 (i.e. the vector $v$ in \cite{bb1}, Corollary 2) and a bound on the number $t$ of linear forms (i.e. rank 1 symmetric tensors). We are finally able to give an answer as complete as possible to these questions. 
The main result of the present paper is the following:

\begin{theorem}\label{a1.1}
Let $P\in \PP N$ with $N={m+d\choose d}-1$.
Suppose that:
$$\begin{array}{c}
\sbr(P)<\sr(P) \hbox{ and}
\\
\sbr(P)+\sr(P)\leq 2d+1.
\end{array}$$
Let $\mathcal{S}\subset X_{m,d}$ be a $0$-dimensional reduced
subscheme that realizes the symmetric rank of $P$, and let
$\mathcal{Z}\subset X_{m,d}$ be a $0$-dimensional non-reduced
subscheme such that $P\in \langle \mathcal{Z} \rangle$ and
$\deg{\mathcal{Z}}\leq \sbr(P)$. There is a unique rational normal curve $C_d\subset X_{m,d}$ such
that $\deg(C_d\cap (\mathcal {S}\cup \mathcal {Z}))\ge d+2$. Then for all points $P\in \PP N$ as above we have that:
$$
\mathcal{S}=\mathcal{S}_1\sqcup \mathcal{S}_2,\ \ \ \ 
\mathcal{Z}=\mathcal{Z}_1\sqcup \mathcal{S}_2,
$$
where $\mathcal{S}_1=\mathcal{S}\cap C_d$,
$\mathcal{Z}_1=\mathcal{Z}\cap C_d$ and
$\mathcal{S}_2=(\mathcal{S}\cap \mathcal{Z})\setminus
\mathcal{S}_1$.
\\
Moreover $C_d$, $\mathcal {S}_2$ and $\mathcal {Z}$ are unique, $\deg (\mathcal {Z}) = \sbr (P)$, $\deg (\mathcal {Z}_1)+\deg (\mathcal {S}_1) = d+2$, $\mathcal {Z}_1\cap \mathcal {S}_1
= \emptyset$ and $\mathcal {Z}$ is  the unique zero-dimensional subscheme $N$ of $X_{m,d}$ such that $\deg (N) \le \sbr(P)$ and $P\in \langle N\rangle$.
\end{theorem}

In the language of polynomials, Theorem \ref{a1.1} can be rephrased as follows.

\begin{corollary}\label{corpoly.1} Let $F\in K[x_0, \ldots , x_m]_d$ be such that
$\sbr(F)+\sr(F)\leq 2d+1$ and $\sbr(F)<\sr(F)$. Then there are an integer $0\leq t\leq (d-1)/2 $, linear forms $L_1, L_2,M_1, \ldots , 
M_t\in K[x_0, \ldots , x_m]_1$,  and a form $Q\in K[L_1,L_2]_d$ such that 
$F=Q+M_1^d+\cdots + M_t^d$, $t \le \sbr(F)+\sr(F)-d-2$, and $\sr(F)=\sr(Q)+t$.
\\
Moreover the line $\langle L_1,L_2\rangle $, the forms $M_1,\dots , M_t$ and $Q$ such that $Q=\sum_{i=1}^q l_i^{d-d_i}m_i$ with $l_i$'s linear forms and $m_i$'s forms of degree $d_i$ such that $\sum (d_i+1)=s-t$, are 
uniquely determined by $F$.
\end{corollary}

An analogous corollary can be stated for symmetric tensors.

\begin{corollary}\label{cortens.1} Let $T\in S^dV^*$ be such that
$\sbr(T)+\sr(T)\leq 2d+1$ and $\sbr(T)<\sr(T)$. Then there are an integer  $0\leq t\leq (d-1)/2 $, vectors $v_1, v_2,w_1, \ldots , 
w_t\in S^1V^*$, and a symmetric tensor $v\in S^d(\langle v_1, v_2 \rangle)$ such that
$T=v+w_1^{\otimes d}+\cdots + w_t^{\otimes d}$, $t\le \sbr(T)+\sr(T)-d-2$, and $\sr(T)=\sr(v)+t$.
Moreover the line $\langle v_1,v_2\rangle $, the vectors $v_1,\dots , v_t$ and the tensor $v$ such that $v=\sum_{i=1}^q u_i^{\otimes (d-d_i)}\otimes z_i$ with $u_i\in \langle v_1, v_2 \rangle$ and $z_i \in S^{d_i}(\langle v_1, v_2 \rangle)$ such that $\sum (d_i+1)=s-t$, are 
uniquely determined by $T$.
\end{corollary}

Moreover,  by introducing the notion of linearly general position of a scheme (Definition \ref{oo1}), we can  perform a finer geometric description of the condition for the uniqueness of the scheme $\mathcal{Z}$ of Theorem \ref{a1.1}. This is the main purpose of Theorem \ref{f2} and Corollary \ref{aaa}.
In terms of homogeneous polynomials and symmetric tensors, they can be phrased as follows:

\begin{corollary}\label{aaaP}
Fix integers $m \ge 2$ and $d\ge 4$. Fix a degree $d$ homogeneous polynomial $F$ in $m+1$ variables (resp. $T\in S^d V$) such that $\sbr(F)\le d$ (resp. $\sbr(T)\le d$). Let $Z\subset \mathbb {P}^m$
be any smoothable zero-dimensional scheme such that $\nu _d(Z)$ evinces $\sbr (F)$ (resp. $\sbr (T)$). Assume that $Z$ is in linearly general position.
Then
$Z$ is the unique scheme which evinces  $\sbr (F)$ (resp. $\sbr (T)$).
\end{corollary}

\section{Proofs}

The existence of a  scheme $\mathcal{Z}$ as in Theorem \ref{a1.1} was known from 
\cite{bgi} and \cite{bgl} (see Remark 1 of \cite{bb1}).

\begin{lemma}\label{oo2}
Fix integers $m \ge 2$ and $d\ge 2$, a line $\ell \subset \mathbb {P}^m$ and any finite set $E\subset \mathbb {P}^m\setminus \ell$ such that $\sharp (E) \le d$. Then $\dim (\langle \nu _d(E )\rangle)
= \sharp (E)-1$ and $\langle \nu _d(\ell )\rangle \cap \langle \nu _d(E)\rangle =\emptyset$.
\end{lemma}

\begin{proof}
Since $h^0(\ell \cup E,\mathcal {O}_{\ell \cup E}(d)) = d+1+\sharp (E)$, to get both statements it is sufficient to prove $h^1(\mathcal {I}_{\ell \cup E}(d))=0$. Let $H\subset \mathbb {P}^m$ be a general
hyperplane containing $\ell$. Since $E$ is finite and $H$ is general, we have $H\cap E = \emptyset$. Hence the residual exact sequence of the scheme $\ell \cup E$ with respect to the hyperplane $H$
is the following exact sequence on $\mathbb {P}^m$:
\begin{equation}\label{eqoo1}
0 \to \mathcal {I}_E(d-1) \to \mathcal {I}_{\ell \cup E}(d) \to \mathcal {I}_{\ell ,H}(d) \to 0.
\end{equation}
Since $h^1(\mathcal {I}_E(d-1)) = h^1(H,\mathcal {I}_{\ell ,H}(d)) =0$, we get the lemma.
\end{proof}

{\emph {Proof of Theorem \ref{a1.1}.}} All the statements are contained in \cite{bb1}, Theorem 1, except the uniqueness of $\mathcal{Z}$, the fact that $\deg (\mathcal {Z}_1)+\deg (\mathcal {S}_1) = d+2$ and  $\mathcal {Z}_1\cap \mathcal {S}_1
= \emptyset$. Let $\ell \subset \mathbb {P}^m$ be the line such that $\nu _d(\ell )=C_d$. Take $Z, S, Z_1,S_1, S_2\subset \mathbb {P}^m$, such that $\nu _d(Z) = \mathcal {Z}$, $\nu_d(S)=\mathcal {S}$, $\nu _d(Z_1)=\mathcal {Z}_1$, and $\nu _d(S_i)
=\mathcal {S}_i$ for $i=1,2$. Assume the existence of another subscheme $\mathcal {Z}' \subset X_{m,d}$
such that $P\in \langle  {\mathcal{Z}}'\rangle$ and $\deg (\mathcal {Z}') \le \sbr(P)$. Set $\mathcal {Z}'_1:= \mathcal {Z}'\cap C_d$. The fact that $\mathcal {Z}' = \mathcal {Z}'_1\sqcup S_2$, is actually the proof  of \cite{bb1}, Theorem 1 (parts (b), (c) and (d)). At the end of step (a) (last five lines) of proof  of \cite{bb1}, Theorem 1, there
is a description of the next steps (b), (c) and (d)  needed to prove that $\mathcal{Z} = (\mathcal{Z} \cap C_d)\sqcup \mathcal{S}_2$ for a certain scheme $\mathcal{Z}$. The role played by $\mathcal{Z}$ in  \cite{bb1}, Theorem 1, is the same that $\mathcal{Z}' $ plays  here, hence the same steps (b), (c) and (d) give
$\mathcal{Z}' = \mathcal{Z}'_1\sqcup \mathcal{S}_2$ as we want here  (one just needs to write $\mathcal {Z}'$ instead of $\mathcal {Z}$).

Since $C_d$ is a smooth curve, $\mathcal {Z}_1\cup \mathcal {Z}'_1\subset C_d$, $\mathcal{S}_2\cap C_d=\emptyset$, and $\mathcal {Z}\cup \mathcal {Z}' =(\mathcal {Z}_1\cup \mathcal {Z}'_1) \sqcup \mathcal{S}_2$, the schemes
$\mathcal {Z}$ and $\mathcal {Z}'$ are
curvilinear. Hence all subschemes of $\mathcal {Z}$ and $\mathcal {Z}'$ are smoothable. Hence any subscheme of either $\mathcal {Z}$ or $\mathcal {Z}'$ may be used to compute the border rank of some point
of $\PP N$.
Since $\deg (\ell \cap (Z\cup S)) \ge d+2$, $\nu _d((Z\cup S)\cap \ell)$ spans $\langle C_d\rangle$. Lemma \ref{oo2} implies  $\langle C_d\rangle \cap \langle \mathcal {S}_2\rangle = \emptyset$.
Since $P\in \langle \mathcal {S}_1\cup \mathcal {S}_2\rangle$
and $\sharp (S) =\sr(P)$, we have $P\notin \langle \mathcal {A}\rangle$ for any $\mathcal {A}\subsetneqq \mathcal {S}$. 
Therefore
we get that $\langle \{P\} \cup \mathcal {S}_2\rangle \cap \langle \mathcal {S}_1\rangle$ is a unique point. Call $P_1$ this point. Similarly,
$\langle \mathcal {Z}_1\rangle \cap \langle \mathcal {S}_2\rangle$ is a unique point and we call it $P_2$. Similarly,
$\langle \mathcal {Z}'_1\rangle \cap \langle \mathcal {S}_2\rangle$ is a unique point and we call it $P_3$.
Since $\langle C_d\rangle \cap \langle \mathcal {S}_2\rangle =\emptyset$, the set  $\langle C_d\rangle \cap \langle \{P\} \cup \mathcal {S}_2\rangle$ is at most one point.
Since $P_i\in \langle C_d\rangle \cap \langle \{P\} \cup \mathcal {S}_2\rangle$, $i=1,2,3$, we have $P_1 = P_2=P_3$ and $\{P_1\} = \langle C_d\rangle \cap \langle \{P\} \cup \mathcal {S}_2\rangle$.
Since $P_1 = P_3$, we have $P_1\in \langle \mathcal {Z}'_1\rangle \cap \langle \mathcal {S}_1\rangle$. Take any $E\subseteq \mathcal {Z}_1$ such that $P_1\in \langle E\rangle$.
Since $P\in \langle \{P_1\}\cup \mathcal {S}_2\rangle \subseteq \langle E\cup \mathcal {S}_2\rangle$ and $P\notin \langle \mathcal {U}\rangle$ for any $\mathcal {U}\subsetneq \mathcal {Z}$, we get $E\cup \mathcal {S}_2=\mathcal {Z}$. Hence $E = \mathcal {Z}_1$.  Therefore
$\mathcal {Z}_1$ computes $\sbr(P_1)$ with respect to $C_d$. Similarly,
$\mathcal {Z}'_1$ computes $\sbr (P_2)$ with respect to the same rational normal curve $C_d$. For any $Q\in \langle C_d\rangle$ with $\sbr (Q) < (d+2)/2$ (equivalently $\sbr (Q) \ne (d+2)/2$), there
is a unique zero-dimensional subscheme of $\langle C_d\rangle$
which evinces $\sbr (Q)$ (\cite{ik}, Proposition 1.36; in \cite{ik}, Definition
1.37, this scheme is called
the canonical form of the polynomial associated to $P$). Since
$P_1=P_2$, we have $\mathcal {Z}'_1 = \mathcal {Z}_1$.
\qed

\begin{definition}\label{oo1}
A scheme $Z\subset \mathbb {P}^m$ is said to be in {\it linearly general position} if for every linear subspace $R\subsetneqq \mathbb {P}^m$
we have $\deg (R\cap Z) \le \dim (R)+1$.
\end{definition}

Notice that the next theorem is false if either $d=2$ or $m=1$. Moreover if $d=3$ and $m>1$, then it essentially says that a point in the tangential variety of a Veronese variety belongs to a unique tangent line. This is a consequence of the  well known Sylvester's theorem on the decompositions of binary forms (\cite{bgi}, \cite{lt}).

\begin{theorem}\label{f2}
Fix integers $m \ge 2$ and $d\ge 4$. Fix $P\in \mathbb {P}^N$. Let $Z\subset \mathbb {P}^m$
be any smoothable zero-dimensional scheme such that $P\in \langle \nu _d(Z)\rangle$ and $P\notin \langle \nu _d(\overline{Z})\rangle$ for any $\overline{Z} \subsetneq Z$. 
Assume $\deg (Z) \le d$ and that $Z$ is in linearly general position.
Then
$Z$ is the unique scheme $Z' \subset \mathbb {P}^m$ such that $\deg (Z') \le d$ and $P\in \langle \nu
_d(Z')\rangle$. Moreover $\nu _d(Z)$ evinces $\sbr (P)$.\end{theorem}

\begin{proof} Since $\deg (Z) \le d$ and $Z$ is smoothable, \cite{bgi}, Proposition 11 (last sentence), gives $\sbr (P) \le d$.  Hence there is a scheme
which evinces $\sbr (P)$ (\cite{bb1}, Remark 3).
The existence of such a scheme follows from \cite{bb1}, Remark 1, and the inequality $\sbr(P)\le d$.
Fix any scheme $Z' \subset \mathbb {P}^m$ such that $Z'\ne Z$, $\deg (Z') \le d$, $P\in \langle \nu _d(Z')\rangle$, and $P\notin \langle \nu
_d(Z'')\rangle$ for any $Z''\subsetneq Z'$. Since $\deg (Z\cup Z') \le 2d+1$ and
$h^1(\mathbb {P}^m,\mathcal {I}_{Z\cup Z'}(d)) >0$ (\cite{bb1}, Lemma 1), there is a line
$D\subset \mathbb {P}^m$ such that $\deg (D\cap (Z\cup Z'))
\ge d+2$ (\cite{bgi}, Lemma 34). Since $Z$ is in linearly general position and $m\ge 2$, we have
$\deg (Z\cap D)\le 2$. Hence $\deg (Z'\cap D)\ge d$. Hence $\deg (Z') = d$.
Since $\deg (Z')=d$, we get
$Z'\subset D$. Hence $P\in \langle \nu _d(D)\rangle$. Hence $\sbr(P) = d$. The secant varieties
of any non-degenerate curve have the expected dimension (\cite{a}, Remark 1.6).
Hence $\sbr (P) \le \lfloor (d+2)/2\rfloor$.  Since $\deg (Z')=d$, we assumed $\deg (Z') \le \sbr (P)$, contradicting the
assumption $d\ge 4$.
\end{proof}

\begin{corollary}\label{aaa}
Fix integers $m \ge 2$ and $d\ge 4$. Fix $P\in \mathbb {P}^N$ such that $\sbr(P)\le d$. Let $Z\subset \mathbb {P}^m$
be any smoothable zero-dimensional scheme such that $\nu _d(Z)$ evinces $\sbr (P)$. Assume that $Z$ is in linearly general position.
Then
$Z$ is the unique scheme which evinces  $\sbr(P)$.
\end{corollary}

\providecommand{\bysame}{\leavevmode\hbox to3em{\hrulefill}\thinspace}

\end{document}